\documentclass{article}
\usepackage[utf8]{inputenc}
\usepackage{hyperref}

\usepackage{amsthm,amssymb,amsmath}
\usepackage{hyperref}
\usepackage{graphicx}

\newtheorem{thm}{Theorem}[section]
\newtheorem{cor}[thm]{Corollary}
\newtheorem{prop}[thm]{Proposition}
\newtheorem{lem}[thm]{Lemma}
\newtheorem{defn}[thm]{Definition}
\newtheorem{rem}[thm]{Remark}

\newtheorem{exmp}[thm]{Example}

\newcommand{\NN}{\mathbb{N}}

\title{On Periodic Points in Covering Systems}
\author{Wang Yihan\\2001110016@pku.edu.cn} 
\date{\today}

\begin{document}

\maketitle

\begin{abstract}
We study a system of intervals $I_1,\ldots,I_k$ on the real line and a continuous map $f$ with $f(I_1 \cup I_2 \cup \ldots \cup I_k)\supseteq I_1 \cup I_2 \cup \ldots \cup I_k$. It's conjectured that there exists a periodic point of period $\le k$ in $I_1\cup \ldots \cup I_k$. In this paper, we prove the conjecture by a discretization method and reduce the initial problem to an interesting combinatorial lemma concerning cyclic permutations. We also obtain a non-concentration property of periodic points of small periods in intervals. 
\end{abstract}

\section{Introduction}

We consider a continuous mapping $f$: $\mathbb{R} \to \mathbb{R}$, and $k$ closed intervals $I_1,I_2,...,I_k$. Throughout this paper, $f^l$ denotes the $l$-th iterate of $f$. And all the mappings are assumed to be continuous unless otherwise noted. 

\begin{defn}
We call $(I_1,I_2,\ldots ,I_k;f)$ a $covering\ system$ if $f(I_1 \cup I_2 \cup \ldots \cup I_k)\supseteq I_1 \cup I_2 \cup \ldots \cup I_k$.
\end{defn}

This concept was first introduced by S. A. Bogatyi and E. T. Shavgulidze in their paper \cite{1988Periodic}, when they tried to obtain an analogue of Sharkovskii's theorem for an arbitrary tree. They proved the existence of a periodic point of period $\le g(n)$ in a covering system and established an upper bound for $g(n)$ $(g(n)\le 2(n^2-1)g(n-1)^{n-2})$. Moreover, they showed $g(n)=n$ for $n\le 4$ and conjectured in the remark that $g(n)=n$ holds for all $n\in \NN$, which is proved in this paper as the main theorem.  

\begin{thm}\label{main}
For any covering system $(I_1,I_2,\ldots,I_k;f)$, there exists $x_0\in I_1 \cup I_2 \cup \ldots \cup I_k$ and $l\le k$ such that $f^l(x_0)=x_0$.
\end{thm}

We will prove this theorem by a discretization method, which reduces the initial problem to a lemma in combinatorics. The main idea is to cut the intervals into small subintervals and perturb the mapping $f$ to get a cyclic permutation. And the proof of the theorem is completed by observing a property of cyclic elements of $S_n$.

The proof of the main theorem \ref{main} will be divided into several steps and discussed in the following sections. We now state a well-known proposition which will be frequently used in the proof of theorem \ref{main}. 

\begin{prop}[the case $k=1$]
Let $f:I \to \mathbb{R}$ satisfy $f(I) \supseteq I$. Then there exists a fixed point of $f$ in $I$.
\end{prop}
This fact is an easy consequence of the intermediate value theorem in calculus so we omit the proof. And we will see later that for $k$ larger, most cases can be reduced to this fundamental one. 

We would like to end this introduction with the statement of the main lemma used in the proof of theorem \ref{main}, which studies the properties of cyclic permutations. Actually, this lemma is an extension of a lemma proved in Sharkovskii's celebrated paper \cite{A1995COEXISTENCE}. See the remark below.

\begin{defn}\label{char}
Let $f$ be a cyclic permutation in the symmetric group $S_n$, $i.e.$, $f$ can be written as $(i_1i_2\ldots i_n)$, meaning that $f: i_1\to i_2\to\ldots\to i_n\to i_1$. And let $A_i=\{i,i+1\}$ $(i=1,2,\ldots,n-1)$ be $(n-1)$ particular sets. We define the characteristic number $m_i$ of each $A_i$ as
$$m_i=\min \{m\arrowvert\ (convf)^m(A_i)\supseteq A_i\}.$$
Here, $convf(A)=$ convex hull of $f(A)=\{\min f(A), \min 
f(A)+1, \ldots,\max f(A)\}$ for any finite set $A\subseteq \mathbb{N}$.\par
And we define the characteristic\ sequence of $f$ to be 
$$m_1'\le m_2'\le \ldots\le m_{n-1}',$$ 
where $\{m_i'\}$ is a rearrangement of $\{m_i\}$.
\end{defn}

\begin{exmp}\label{m2}
If $f=(136245)\in S_6$, then we can calculate $m_1=m_3=m_5=3$, $m_2=2$, $m_4=1$. The characteristic sequence is $1\le2\le3\le3\le3$.\par 
Take $m_2$ as an example. Since $f(A_2)=f(\{2,3\})=\{4,6\}$, $convf(A_2)=conv\{4,6\}=\{4,5,6\}$; and $f(\{4,5,6\})=\{1,2,5\}$, $conv\{1,2,5\}=\{1,2,3,4,5\}$; therefore, $(convf)^2(A_2)=\{1,2,3,4,5\}\supseteq A_2$ and we have $m_2=2$. 
\end{exmp}

Now we can state our main lemma.

\begin{lem}\label{chars}
For any cyclic permutation $f$ in the symmetric group $S_n$, the characteristic sequence $\{m_i'\}$ of $f$ satisfies $m_i'\le i$, \ $i=1,2,\ldots,n-1$.
\end{lem}

\begin{rem}
(1) For another form of this lemma more closely related to the initial problem, see lemma \ref{dis} in the next section. \par
(2) In theorem 7 of Sharkovskii's paper \cite{A1995COEXISTENCE}, it is actually proved that $m_i'\le n-1$, for $i=1,2,\ldots,n-1$. Therefore, this lemma is a generalization of that result.\par 
Meanwhile, this lemma leads in some sense to a non-concentration property of periodic points with lower periods. Namely, it follows that if $x_{k_1}<x_{k_2}<\ldots<x_{k_n}$ is an orbit of periodic points of period $n$ under $f$, then between any two adjacent points $x_{k_i}$ and $x_{k_{i+1}}$, there exists a periodic point $y_i$ of period $m_i\le n-1$. Futhermore, these $(n-1)$ numbers $m_i$ can be rearranged to have $m_i'\le i$. In other words, all the periodic points of period $\le i$ cannot lie only in $(i-1)$ of the intervals $[x_{k_j},x_{k_{j+1}}]$, $1\le j\le n-1$.\par      
(3) One can see that the statement of lemma \ref{chars} doesn't involve any notations in analysis. And hence it is of independent interest in the study of the symmetric group $S_n$ in combinatorics.\par
(4) It is necessary to require the permutation be cyclic. As a counterexmaple, consider $f=(13)(2)\in S_3$, then the characteristic numbers are $m_1=m_2=2$.
\end{rem}

 We organize the paper as follows. In section 2, we introduce the discretization method and reduce theorem \ref{main} in dynamics to a lemma in combinatorics. Then in section 3, we prove the equivalence of the lemma in section 2 and the one stated above. Finally in section 4, we prove lemma \ref{chars} .

\section{Discretization}

In this section we show that Theorem \ref{main} can be proved by solving another discrete problem related to it. The main idea is to divide the initial intervals into smaller subintervals, such that the image of each subinterval under $f$ contains the whole of some other subintervals, not just part of them. However, it may happen that no matter how you divide the intervals, there will be an interval whose image under $f$ contains only part of another interval. What saves us here is that, we can consider perturbation of $f$ if the intervals are cut into enough small pieces. Indeed, we will use the following lemma:

\begin{lem}\label{ptb}
If there is a covering system $(I_1,I_2,\ldots,I_k;f)$ such that $\forall x \in I_1\cup \ldots\cup I_k$, and $\forall\: l\le k, f^l(x)\neq x,$ then $\exists\: \delta >0$ depending on $f$, $s.t.$, whenever $\lVert g-f\rVert_{C(\mathbb{R})}<\delta$, it also holds that $g^l(x)\neq x$, for all $l\le k$, $x\in I_1\cup \ldots \cup I_k$.
\end{lem}
\begin{proof}
Recall that $f$ has a periodic point of period $\le k$ in $I_1\cup \ldots\cup I_k$ is equivalent to that $F(x)=(f(x)-x)\ldots(f^k(x)-x)$ has a zero in $I_1\cup \ldots\cup I_k$. Therefore, we only need to control $\lvert f^l(x)-g^l(x) \rvert$ for $1\le l\le k$, $x\in I_1\cup \ldots\cup I_k$. And this will be done inductively. Indeed, we can write:
$$f^{l+1}(x)-g^{l+1}(x)=f(f^l(x))-f(g^l(x))+f(g^l(x))-g(g^l(x))$$
And if we assume we have proved $\lvert f^l(x)-g^l(x)\rvert$ is sufficiently small (at least less than 1) for all $x\in I_1\cup \ldots \cup I_k$, then choose a closed interval $[p,q]$ containing $f^l(I_1\cup \ldots \cup I_k)$. We will have $g^l(I_1\cup \ldots \cup I_k)\subseteq [\,p-1,\,q+1\,]$. Thus, we will obtain a control of $\lvert f^{l+1}(x)-g^{l+1}(x) \rvert$ when $x\in I_1\cup \ldots \cup I_k$, since $f$ is uniformly continuous on $[\,p-1,\,q+1\,]$ and $\lVert g-f\rVert_{C(\mathbb{R})}<\delta$. Finally, we repeat the procedure $k$ times and get the conclusion.
\end{proof}

Now we can introduce the discrete problem related to Theorem \ref{main}. To begin with, we make a technical assumption on the initial covering system in Lemma \ref{ptb}. (See remark \ref{ass} for the reason). Namely, if $I_j=[a_j,b_j]$ and $f(I_j)=[c_j,d_j]$ in the given covering system $(I_1,I_2,\ldots,I_k;f)$, then we can assume $f(\{a_j,b_j\})=\{c_j,d_j\}$ without loss of generality, since we can replace $I_j$ by a smaller subinterval if necessary.\par
Let $a_1<b_1<a_2<b_2<\ldots<a_k<b_k$ be all the end points of the intervals $I_1,\ldots,I_k$ and denote $M_0=\{a_1,b_1,\ldots,a_k,b_k\}$. Consider their images $f(M_0)=\{ f(a_1),\ldots,f(b_k)\}$ and we eliminate the points which don't belong to $I_1\cup \ldots\cup I_k$ to obtain a set $S_1\subseteq f(M_0)$. Put $M_1=M_0\cup S_1$. Inductively, in each step we consider the image of $M_i$ under the mapping $f$ and we eliminate those points dropping out of $I_1\cup ...\cup I_k$ to get a set $S_{i+1}\subseteq f(M_i)$. Then let $M_{i+1}=M_i\cup S_{i+1}$. 

Finally we obtain a sequence of sets $M_0\subseteq M_1\subseteq \ldots \subseteq M_i\subseteq \ldots$, where each $M_i$ contains the images of the initial end point set $M_0$ under \{$id,f,f^2,\ldots,f^i$\} except for those dropping out of $I_1\cup \ldots\cup I_k$.

Then we can find a sufficiently large integer $N$ such that 
$\forall x\in M_N-M_{N-1}$, $dist(x, M_{N-1})<\delta$, where $\delta$ comes from Lemma \ref{ptb}. Such $N$ exists because if not, there will be a sequence of disjoint points in $[a_1,b_k]$ with pairwise distance larger than $\delta$, which is absurd.  Therefore, we can take small perturbation of $f$ to get a map $\tilde{f}$ with $\lVert \tilde{f}-f\rVert<\delta$, and $\widetilde{M_{N-1}}=\widetilde{M_N}$ for $\tilde{f}$. Indeed, we just move each point $(x_0,f(x_0))\in M_{N-1}\times M_{N}$ on the graph of $f$ to the nearest point $(x_0,y_0)\in \{x_0\}\times M_{N-1}$ preserving the continuity and the covering property of $f$. For example, in a neighborhood of each $x_0\in M_N$, we can define $\tilde{f}(x)=f(x)\pm\sigma\phi(x-x_0)$, where $0\le\phi\le1$ is a cut-off function, $\phi(0)=1$, $\sigma<\delta$ is chosen so that $\tilde{f}(x_0)\in M_{N-1}$. And for different $x_0$ we choose different signs $\pm$ so that $(I_1,I_2,\ldots,I_k;\tilde{f})$ is still a covering system. (See Figure \ref{fig:my_label1} and \ref{fig:my_label2}.)\par
Thus, if we divide the intervals into small pieces using the points in $\widetilde{M_N}$, then $\tilde{f}$ will have the nice property that the image of each subinterval contains entirely another one or more subintervals. In other words, $\tilde{f}$ can be regarded as a discrete mapping from $\{1,2,\ldots,n\}$ to \{subsets of $\{1,2,\ldots,n\}$\}. Here each number $j$ represents a small piece of interval $J_j$. And $j'\in f(j)$ if $f(J_j)\supseteq J_{j'}$. Note that the image of some interval may be empty, since we ignore the part outside $I_1\cup \ldots\cup I_k$. And we only care about the interval determined by $\{f(a),f(b)\}$ if $a,b\in \widetilde{M_{N}}$, although the image $f([a,b])$ may be larger.

\begin{exmp}
The two figures illustrate how we perturb $f$ locally to obtain $\tilde{f}$ with $\tilde{f}$ still continuous and inducing a covering system close to the initial one.
\begin{figure}[ht]
    \centering
    \includegraphics[width=\textwidth]{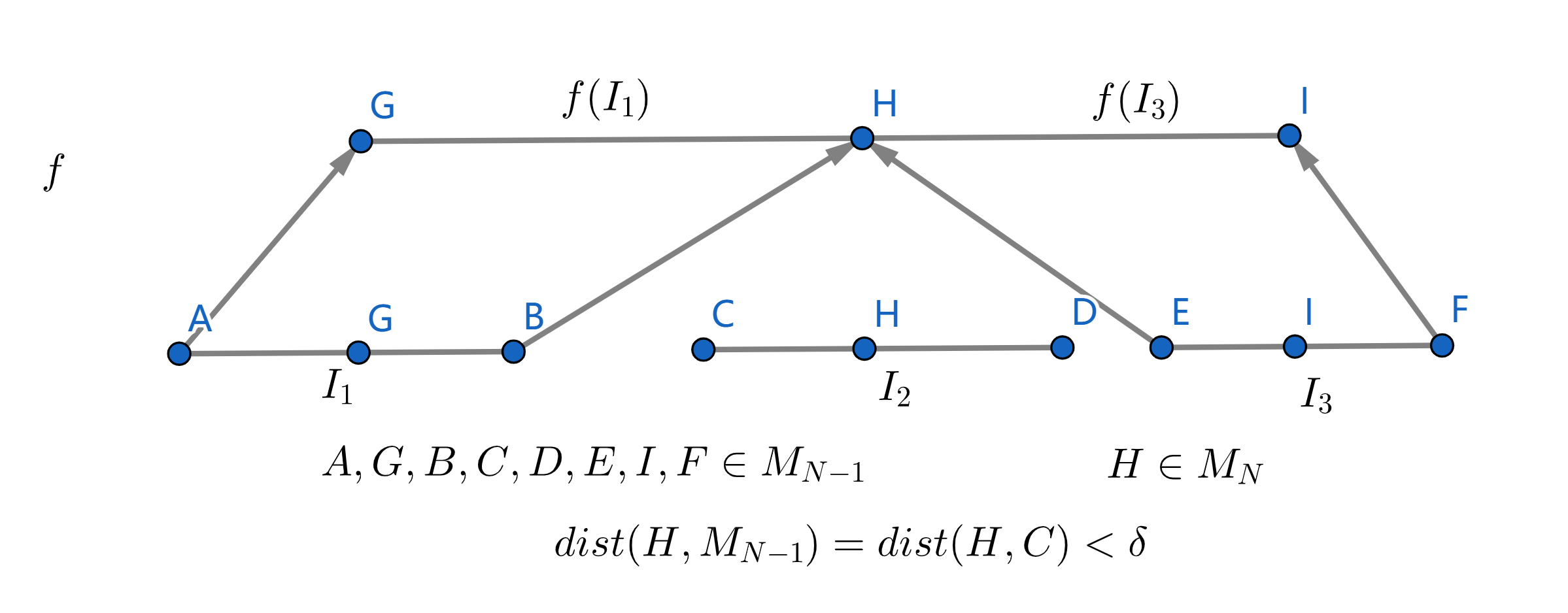}
    \caption{the initial $f$}
    \label{fig:my_label1}
\end{figure}
\begin{figure}[ht]
    \centering
    \includegraphics[width=\textwidth]{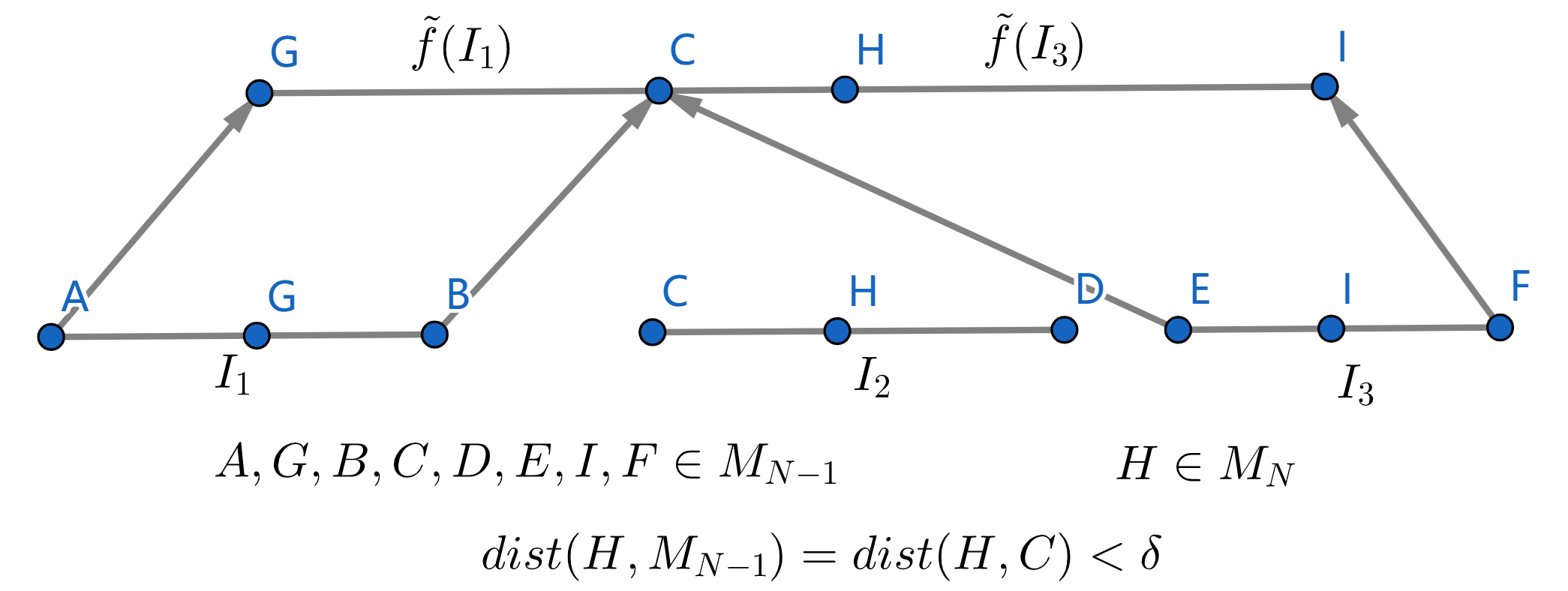}
    \caption{the construction of $\tilde{f}$}
    \label{fig:my_label2}
\end{figure}
\end{exmp}

It should also be remarked here that the only useful information of $f$ on the so called ``gaps" (see \cite{1988Periodic}) between two adjacent intervals $I_i$ and $I_{i+1}$ is the behavior of the end points of each gap. In other words, for the image $f(J_j)$ of each small subinterval $J_j$, we only focus on the part $f(J_j)\cap (I_1\cup\ldots\cup I_k)$. This is the reason why we eliminate the points outside $I_1\cup\ldots\cup I_k$ in each step. We will see later that this information is enough for the proof.

To sum up, once we have a mapping $f$ satisfying the conditions in Lemma \ref{ptb}, we can find an $\tilde{f}$ close enough to $f$, which also satisfies the conditions but can be viewed as a discrete mapping. Such $f$ and $\tilde{f}$ are counterexamples to Theorem \ref{main}. Therefore, if we argue by contradiction, we can easily see that Theorem \ref{main} is a corollary of the following lemma, whose proof will be discussed later. 

\begin{lem}\label{dis}
For any mapping $f$: $\{1,2,\ldots,n\}$ $\to$ \{subsets of $\{1,2,\ldots,n\}\}$ such that $\bigcup_{i=1}^n f(i)=\{1,2,\ldots,n\}$, and for any partition 
$$I_1=\{1,2,\ldots,i_1\}, I_2=\{i_1+1, i_1+2,\ldots,i_2\}, \ldots, I_k=\{i_{k-1}+1,\ldots,n\},$$
there exists $j\in \{1,2,\ldots,k\}$ and $r,s\in I_j$ ($r=s$ is allowed) such that
$$(convf)^l(\{r,s\})\supseteq \{r,s\},\quad for\ some \quad l\le k.$$
Here, $convf(A)=$ convex hull of $f(A)=\{\min f(A), \min f(A)+1,\ldots,\max f(A)\}$ for any finite set $A\subseteq \mathbb{N}$.   
\end{lem}

\begin{rem}\label{ass}
We emphasize here again the difference between $f([a,b])$ and $[f(a),f(b)]$. In general the former is larger. And one may worry that the discrete mapping $f$ induced from a covering system does not satisfy the corresponding covering property $\bigcup_{i=1}^n f(i)=\{1,2,\ldots,n\}$, since in the definition of $f$, $j'\in f(j)$ iff $f(J_j)\supseteq J_{j'}$, where $f(J_j)$ is understood in the sense of $[f(a),f(b)]$ rather than $f([a,b])$. However, this problem can be easily handled by considering the initial covering system to be minimal. (See \cite{1988Periodic} for more details.) That is, in particular, the end points of each $I_j$ are mapped to the end points of $f(I_j)$. And it's now easily checked that the perturbation and discretization preserve the covering property of $f$, since the perturbation does not change the position of the end points of $f(I_j)$ essentially.       
\end{rem}

\begin{proof}[Proof of Theorem \ref{main} by lemma \ref{dis}]
Suppose the covering system $(I_1,\ldots,I_k;f)$ is a counterexample to Theorem \ref{main}. Then as mentioned above, after perturbation, we can assume $f$ satisfies the conditions of lemma \ref{dis}. Therefore the conclusion of lemma \ref{dis} gives us a subinterval of $I_j$ represented by $conv\{r,s\}$ whose image under $f^l$ contains itself. Thus there must exist a periodic point of period $\le l\le k$ in $I_j$, a contradiction.
\end{proof}

\section{Equivalence of the Two Lemmas}

Although lemma \ref{dis} reduces Theorem \ref{main} to a discrete problem, it is not satisfactory since it involves something inconvenient to deal with, like convex hull and partition. In the following we will simplify the conditions in lemma \ref{dis} and find another description of it which only depends on the own property of a permutation. And the main goal of this section is to prove the equivalence of lemma \ref{chars} and lemma \ref{dis}.

Let $f$ be in lemma \ref{dis}. We can assume further:\par
(1)$\quad \forall i\neq j$, $f(i)\cap f(j)= \emptyset$.\par
(2)$\quad \forall i$, $f(i) \neq \emptyset$.\par
(3)$\quad \forall i$, $f(i)$ contains exactly one element.\par
(4)$\quad f$ is a cyclic permutation in the symmetric group $S_n$, $i.e.$, $f$ can be written as $(i_1i_2\ldots i_n)$.\par
Indeed, (1) can be realized since we can move out the common part of $f(i)$ and $f(j)$ from one of them without changing the property $\bigcup_{i=1}^n f(i)=\{1,2,\ldots,n\}$.\par
(2) is also satisfied because we can eliminate the number $i$ and restrict the mapping $f$ to $\{1,2,\ldots,n\}-\{i\}$ if necessary.\par
(3) is a consequence of (1), (2) and the condition $\bigcup_{i=1}^n f(i)=\{1,2,\ldots,n\}$. Indeed, one can apply the operations in (1) and (2) repeatedly until a minimal case. This couldn't be an endless process because we are dealing with finite sets. And one can prove that the minimal set cannot be empty because the covering property is preserved.\par
For (4), note that $f \in S_n$ because of (3). Suppose $f$ is not cyclic. Then choose an orbit $(i_1\ldots i_m)$ of $f$ ($m<n$) and restrict both $f$ and the partition $I_1,\ldots,I_k$ to $\{i_1,\ldots,i_m\}$.

In conclusion, we obtain:
\begin{prop}
It is sufficient to consider $f\in S_n$ and $f$ being cyclic in lemma \ref{dis}.
\end{prop}

Now we turn to the proof of the equivalence of lemma \ref{dis} and lemma \ref{chars}. (See also Definition \ref{char}.)

\begin{prop}\label{imp}
Lemma \ref{chars} implies lemma \ref{dis}.
\end{prop}
\begin{proof}
As mentioned above, we can assume $f\in S_n$ being cyclic, without loss of generality. \par
Suppose lemma \ref{chars} is true. Then for any partition 
$$I_1=\{1,2,\ldots,i_1\},\quad I_2=\{i_1+1,\ldots,i_2\},\quad \ldots,\quad I_k=\{i_{k-1}+1,...,n\},$$
the $(k-1)$ characteristic numbers $\{m_{i_1},m_{i_2},\ldots,m_{i_{k-1}}\}$ cannot cover $\{m_1',\ldots,m_k'\}$. Therefore, $\exists\, t\neq i_1,\ldots,i_{k-1}$, and $l\le k$ such that $m_t=m_l'\le k$, since $m_1'\le \ldots\le m_k'\le k$ in the characteristic sequence. In other words, 
$$(convf)^{m_t}(A_t)\supseteq A_t,\quad where\quad A_t=\{t,t+1\}\subseteq some\ I_j.$$
Therefore, we can just choose $r=t,\ s=t+1$ in lemma \ref{dis}.\par
\end{proof}

Recall that in one-dimensional dynamics we often consider a directed graph associated to a periodic orbit (For more details, see \cite{Uhland1982Interval} or chapter 1 of \cite{onedimdynamics1992}). Namely, if $f$ is a cyclic permutation in Definition \ref{char}, we can construct a directed graph $\Gamma_f$ with $(n-1)$ vertices $\{A_1,A_2,\ldots,A_{n-1}\}$, where $A_i=\{i,i+1\}$ is defined in \ref{char}. And there is an edge from $A_i$ to $A_j$ if and only if $convf(A_i)\supseteq A_j$.

\begin{exmp}\label{graph}
Let $f=(136245)\in S_6$ as in example \ref{m2}. Then the associated directed graph $\Gamma_f$ is drawn in figure \ref{fig:my_label3}.
\begin{figure}[ht]
    \centering
    \includegraphics{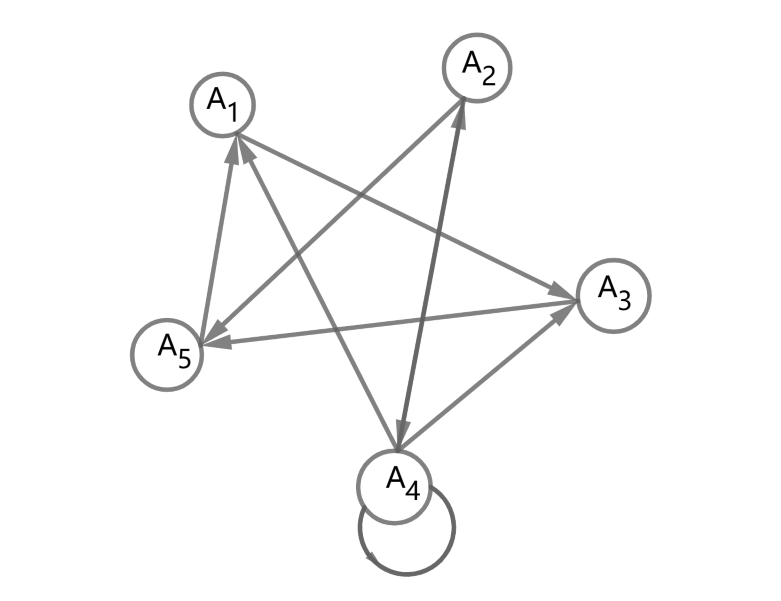}
    \caption{the directed graph $\Gamma_f$}
    \label{fig:my_label3}
\end{figure}
\end{exmp}

Thus, we see directly from the definition that the following proposition holds:

\begin{prop}\label{gr}
For a cyclic $f\in S_n$ and its associated directed graph $\Gamma_f$, the characteristic number $m_i$ is equal to the minimal length of cycles starting from $A_i$ in $\Gamma_f$.
\end{prop}

Using directed graphs, we can complete the proof of the equivalence of the two lemmas.

\begin{prop}
Lemma \ref{chars} is equivalent to lemma \ref{dis}.
\end{prop}

\begin{proof}
Suppose lemma \ref{dis} is true. For a cyclic $f\in S_n$,  suppose on the contrary that there is a $k\le n-1$ such that $m_k'\geq k+1$, and $k$ is the minimal number with this property. Then there are only $(k-1)$ characteristic numbers $m_{i_1},m_{i_2},\ldots,m_{i_{k-1}}\le k$. Without loss of generality, assume $i_1<\ldots<i_{k-1}$. And consider the partition:
$$I_1=\{1,2,\ldots,i_1\},\quad I_2=\{i_1+1,\ldots,i_2\},\quad \ldots,\quad I_k=\{i_{k-1}+1,\ldots,n\}.$$
By lemma \ref{dis}, we can find $j,l\le k$ and $r,s\in I_j$ ($r<s$) such that
$$(convf)^l(\{r,s\})\supseteq \{r,s\}.$$
Note that $r=s$ cannot happen in this case, since $m_r=m_s\geq k+1$. Now if we extend $f$ to be a piecewise linear mapping from the interval $[1,n]$ to $[1,n]$, then we will obtain $f^l([r,s])\supseteq [r,s]$. Therefore, there exists $x_0\in [r,s]$ with $f^l(x_0)=x_0$. Note that since $f$ is cyclic as a permutation in $S_n$ and $l\le k\le n-1$, $x_0$ cannot belong to $\{1,2,\ldots,n\}$. Consequently, $x_0\in (t,t+1)$ for some integer $t\in [r,s-1]$  and it is similar for $f(x_0),f^2(x_0),\ldots,f^{l-1}(x_0)$. Since $f$ is piecewise linear, $x_0\in (t,t+1)$, $f(x_0)\in (t',t'+1)$ imply $f([t,t+1])\supseteq [t',t'+1]$, and it is similar for other pairs of points. So the orbit $\{x_0,f(x_0),\ldots,f^{l-1}(x_0),f^l(x_0)\}$ gives a cycle of length $l\le k$ starting from the vertex $A_t$ in the associated directed graph $\Gamma_f$, which implies $m_t\le k$. But this $t$ together with $i_1,\ldots,i_{k-1}$ gives $k$ characteristic numbers $\le k$, contradicting $m_{k}'\geq k+1$.\par
Combined with Proposition \ref{imp}, the proof is completed.
\end{proof}

Therefore, we have reduced the initial problem on the existence of periodic points to a lemma in combinatorics on the properties of cyclic permutations. In the next section, we will prove the lemma and that will complete the proof of theorem \ref{main}.

We end this section with some examples of cyclic permutations whose characteristic sequences are known.
\begin{exmp}
(1) Let $f\in S_n$ be $(1\to 2\to \ldots\to n\to 1)$. And more generally, $f(i)\equiv i+m\pmod{n}$, where $(m,n)=1$. Then the characteristic sequence of $f$ is $1\le 2\le \ldots\le n-1$.\par
(2) Let $f\in S_{2n+1}$ be of Stefan type (see \cite{A1995COEXISTENCE} or \cite{P1977A} for more details), that is, 
$$1\to (n+1)\to (n+2)\to n\to (n+3)\to (n-1)\to\ldots\to 2n \to 2\to (2n+1)\to1.$$
Then the characteristic sequence of $f$ is $$1\le 2\le 2\le 4\le 4\le\ldots\le 2n-2\le 2n-2\le 2n.$$\par
(3) In \cite{10.2307/2690145}, the author gives all the directed graphs associated to cyclic permutations in $S_5$.
\end{exmp}

\section{Proof of the Lemma in Combinatorics}

We finally prove lemma \ref{chars} in this section. We first recall some definitions and notations from one-dimensional dynamics and linear algebra.

\begin{defn}
Let $f\in S_n$ be a cyclic permutation and $\Gamma_f$ be its associated directed graph. (See the discussion before propositon \ref{gr}.) We can define $T_f$ to be the $(n-1)\times (n-1)$ adjacency matrix of $f$ and $\Gamma_f$. That is, $T_{ij}=1$, if $convf(A_j)\supseteq A_i$, or equivalently, if there is an edge from vertex $A_j$ to vertex $A_i$ in $\Gamma_f$. Otherwise, $T_{ij}=0$.
\end{defn}

\begin{exmp}
Let $f=(136245)\in S_6$ be the permutation in example \ref{m2} and example \ref{graph}. Then the adjacency matrix $T_f$ is
$$
\begin{bmatrix}
 &0 & 0 & 0 & 1& 1&\\
 &0 & 0 & 0 & 1& 0&\\
 &1 & 0 & 0 & 1& 0&\\
 &0 & 1 & 0 & 1& 0&\\
 &0 & 1 & 1 & 0& 0&
\end{bmatrix}
$$
\end{exmp}

We denote the field $\{0,1\}$ with exactly two elements by  $\mathbb{F}_2$. And we say $A\equiv B\pmod{2}$ for integer valued matrices $A$ and $B$, when $A=B$ as matrices in $M(\mathbb{F}_2)$.

\begin{prop}\label{mod}
For any cyclic permutation $f\in S_n$ and $l\in \NN_+$, it holds that $T_{f}^l\equiv T_{f^l}\pmod{2}$.
\end{prop}
A proof of this proposition can be found in \cite{onedimdynamics1992}, Chapter 1, Propositon 20. And for readers' convenience, we give a sketch of proof here.
\begin{proof}[Sketch of proof]
Given $f=(i_1i_2\ldots i_n)\in S_n$ cyclic, we first extend $f$ to be a continuous piecewise linear mapping from the closed interval $[1,n]$ to itself. And recall we define $A_i=\{i,i+1\}$. Later on we will also regard $A_i$ as the closed interval $[i,i+1]$.\par
Now by the definition of $T_f$ and the continuity of $f$, $(T_f)_{ij}=1$ iff $f(A_j)\supseteq A_i$. Consider the interval $A_i=[i,i+1]$ and its image $f^l(A_i)$. In general, we cannot expect $f^l(A_i)=[f^l(i),f^l(i+1)]$ to hold, where $[a,b]$ denotes the closed interval whose end points are $a$ and $b$. We can only say that $f^l(A_i)\supseteq[f^l(i),f^l(i+1)]$ for continuous $f$. However, if we count the multiplicity of each point in the image $f^l(A_i)$, we will find that $f^l(A_i)=[f^l(i),f^l(i+1)]$ actually holds in the sense of mod 2. Indeed, $f^l(A_i)$ can be viewed as a continuous curve in $\mathbb{R}$ starting from the point $f^l(i)$ and ending at $f^l(i+1)$. Therefore, every point $p\in f^l(A_i)$ is counted an even number of times in this curve, unless $p\in[f^l(i),f^l(i+1)]$.\par 
Finally recall the definition of $T_f$ and $T_{f^l}$, and we get the conclusion of the proposition in the sense of mod 2.
\end{proof}

\begin{rem}
If the readers are familiar with representation theory of the symmetric group $S_n$, then they may find that $T_f$ defined above is just the matrix of the tautological representation of $f$ mod 2.
\end{rem}

\begin{cor}\label{id}
For cyclic $f\in S_n$, we have $T_f^n\equiv I\pmod{2}$, where $I=I_{n-1}$ is the unit matrix. 
\end{cor}

\begin{thm}\label{det}
For cyclic $f\in S_n$, the characteristic polynomial of $T_f$, viewed as the determinant of a matrix in $M_{n-1}(\mathbb{F}_2)$, is exactly:
$$\det(\lambda I-T_f)=1+\lambda+\lambda^2+\ldots+\lambda^{n-1}\pmod{2}.$$
\end{thm}

We need the following lemma:
\begin{lem}
There exists a vector $\alpha \in \mathbb{F}_2^{n-1}$ such that $\alpha, T_f\alpha,\ldots, T_f^{n-2}\alpha$ are linearly independent over the field $\mathbb{F}_2$.    
\end{lem}
\begin{proof}
We claim that if $f\in S_n$ is written as $(i_1i_2\ldots i_n)$, then the vector $\alpha=(0,\ldots,0,1,\ldots,1,0,\ldots,0)^T=\sum_{j=i_1}^{i_2-1}e_j$ (or $\sum_{j=i_2}^{i_1-1}e_j$) satisfies the desired property, where $e_j\in \mathbb{F}_2^{n-1}$ is the unit vector whose $j$-th component is $1$.

We prove the claim by induction on $n$.
The case $n=2$ is trivial. We assume the statement holds for all natural numbers $<n$, and we will prove it for $n$.

Suppose $$\mu_1\alpha+\mu_2T_f\alpha+\ldots+\mu_{n-1}T_f^{n-2}\alpha=0.$$
By the definition of $\alpha$, $f=(i_1i_2\ldots i_n)\in S_n$, as well as the argument in the proof of proposition \ref{mod}, we can see that 
$$T_f^m\alpha=\sum_{j=i_{m+1}}^{{i_{m+2}-1}}e_j,\quad or \sum_{j=i_{m+2}}^{{i_{m+1}-1}}e_j,\quad \forall m\le n-2.$$
In other words, if we view each vector $e_j$ as the closed interval $[j,j+1]$ it represents, then we find that
$$\alpha=[i_1,i_2],\quad T_f\alpha=[i_2,i_3],\quad\ldots\quad, T_f^{n-2}\alpha=[i_{n-1},i_n].$$
Now if $i_1=1$, take the inner product with $e_1$, and we obtain:
$$0=\left(\mu_1\alpha+\mu_2T_f\alpha+\ldots+\mu_{n-1}T_f^{n-2}\alpha,\: e_1\right)=\mu_1,$$
because $i_2,\ldots,i_n\geq 2$ and hence $(T_f^m\alpha, e_1)=0$ for $1\le m\le n-2$. Similarly we can obtain $\mu_1=0$ when $i_1=n$.\par
If $2\le i_1\le n-1$, then since $i_1$ does not appear as the end points of $T_f^m\alpha$, $[i_1-1,i_1]$ and $[i_1,i_1+1] \subseteq or \not\subseteq T_f^m\alpha$ at the same time, where $m\geq1$. Or equivalently,
$$(T_f^m\alpha, e_{i_1-1})=(T_f^m\alpha, e_{i_1}),\quad 1\le m\le n-2.$$
And then if we take the inner product with $(e_{i_1-1}+e_{i_1})$, we obtain:
$$0=\left(\mu_1\alpha+\mu_2T_f\alpha+\ldots+\mu_{n-1}T_f^{n-2}\alpha,\:(e_{i_1-1}+e_{i_1})\right)=\mu_1\pmod{2},$$
Anyway, we find that $\mu_1=0$ and consequently,
$$\mu_2T_f\alpha+\ldots+\mu_{n-1}T_f^{n-2}\alpha=0.$$
Now if we put $g=(i_2i_3\ldots i_n)\in S_{n-1}$ and $\beta=T_f\alpha$ with the $i_1$-th component dropped, we can see that $T_f^m\alpha=T_g^{m-1}\beta$, regardless of the $i_1$-th component of $T_f^m\alpha$. This is easily seen to be true since one can find $T_f^m\alpha=T_g^{m-1}\beta=[i_{m+1},i_{m+2}]$ as intervals. And the following holds for $\beta$ and $T_g$:
$$\mu_2\beta+\mu_3T_g\beta+\ldots+\mu_{n-1}T_g^{n-3}\beta=0.$$
By induction hypothesis we conclude that $\mu_2=\ldots=\mu_{n-1}=0$. Therefore, we have proved that the claim also holds for $n$. This completes the proof of the lemma.
\end{proof}

\begin{cor}\label{mc}
The minimal polynomial $g(\lambda)$ of $T_f$ has degree at least $n-1$, and therefore equals the characteristic polynomial.
\end{cor}
\begin{proof}
Suppose $\deg g(\lambda)\le n-2$, then we have
$g(T_f)\alpha=0$, which contradicts the lemma above. And the second statement follows by Cayley-Hamilton's theorem. (Note that $T_f$ is an $(n-1)\times (n-1)$ matrix.)
\end{proof}

With these results, we can compute the characteristic polynomial of $T_f$.
\begin{proof}[Proof of Theorem \ref{det}]
 It follows from corollary \ref{id} and \ref{mc} that 
 $$\det(\lambda I-T_f)=g(\lambda)\mid(\lambda^n-1)=(1+\lambda)(1+\lambda+\lambda^2+\ldots+\lambda^{n-1})\pmod{2}.$$
 Moreover, since the only degree-one irreducible polynomials are $\lambda$ and $\lambda+1$ in $\mathbb{F}_2[\lambda]$, the right hand side can be further decomposed as
 $$1+\lambda+\lambda^2+\ldots+\lambda^{n-1}=(1+\lambda)^mh(\lambda),$$
 where $m\geq0$, and $h(\lambda)$ is a product of several irreducible polynomials of degree $\geq 2$. Hence, by comparing the degree we conclude that 
 $$\det(\lambda I-T_f)=1+\lambda+\lambda^2+\ldots+\lambda^{n-1}\pmod{2}.$$
\end{proof}

\begin{proof}[Proof of lemma \ref{chars}]
We show that if the coefficient of $\lambda^{n-1-i}$ in $\det(\lambda I-T_f)$ is nonzero, then $m_i'\le i$ holds for the characteristic sequence in lemma \ref{chars}.

Let us consider how the coefficients in $\det(\lambda I-T_f)$ are computed. It is then easily seen that the coefficient of $\lambda^{n-1-i}$ is nonzero implies the existence of an $i\times i$ principle matrix $P_i$ of $T_f$ whose determinant is nonzero. It then follows that there exists at least one diagonal of $P_i$ whose elements are all equal to 1. Now it is easily checked that this diagonal give rise to several cycles of length $\le i$ in the associated directed graph $\Gamma_f$. And the vertices of these cycles correspond to the subscripts of the columns of $P_i$. Thus there exist at least $i$ vertices in $\Gamma_f$, from whom the minimal length of cycles starting are all $\le i$. By proposition \ref{gr}, this means that at least $i$ of the characteristic numbers $\le i$. So $m_i'\le i$.
\end{proof}

\section*{Acknowledgement}
I would like to thank Prof. S. A. Bogatyi for introducing the problem to me. And I thank my supervisor, Prof. Tian Gang, for his encouragement as well as helpful suggestions. I'm also grateful to Dr. Ye Yanan and Dr. Zhao Yikai for writing computer programs to verify some claims.

\bibliographystyle{unsrt}
\bibliography{ref}

\end{document}